\documentclass[11pt,a4paper]{amsart} 
\usepackage{amssymb}
\usepackage[all]{xy}
\SelectTips{cm}{}
\calclayout 

\title{Notes on the Gabriel-Roiter measure
} 
\author[Henning Krause]{Henning Krause}
\address{Henning Krause\\ Fakult\"at f\"ur Mathematik\\
Universit\"at Bielefeld\\ 33501 Bielefeld\\ Germany.}
\email{hkrause@math.uni-bielefeld.de}

\newtheorem*{lem}{Lemma}

\newtheorem*{prop}{Proposition}
\newtheorem*{cor}{Corollary}
\newtheorem*{thm}{Theorem}

\theoremstyle{remark}

\theoremstyle{definition}
\newtheorem*{exm}{Example}
\newtheorem*{defn}{Definition}
\numberwithin{equation}{subsection}
\newtheorem*{rem}{Remark}

\renewcommand{\mod}{\operatorname{mod}\nolimits}

\newcommand{\smatrix}[1]{\left[\begin{smallmatrix}#1\end{smallmatrix}\right]}

\renewcommand{\leq}{\leqslant}
\renewcommand{\geq}{\geqslant}
\newcommand{\Ch}{\operatorname{Ch}\nolimits}
\newcommand{\ind}{\operatorname{ind}\nolimits}

\newcommand{\id}{\operatorname{id}\nolimits}

\newcommand{\Hom}{\operatorname{Hom}\nolimits}

\newcommand{\Ker}{\operatorname{Ker}\nolimits}
\newcommand{\Coker}{\operatorname{Coker}\nolimits}

\newcommand{\soc}{\operatorname{soc}\nolimits}
\renewcommand{\dim}{\operatorname{dim}\nolimits}

\newcommand{\Ext}{\operatorname{Ext}\nolimits}

\newcommand{\op}{\mathrm{op}}

\newcommand{\inc}{\mathrm{inc}}
\newcommand{\init}{\mathrm{init}}
\newcommand{\cent}{\mathrm{cent}}
\newcommand{\term}{\mathrm{term}}

\newcommand{\lto}{\longrightarrow}

\def\a{\alpha}
\def\b{\beta}

\def\g{\gamma}
\def\p{\phi}

\def\m{\mu}

\def\la{\lambda}

\def\La{\Lambda}

\def\A{{\mathcal A}}

\def\C{{\mathcal C}}

\def\bbN{\mathbb N}

\def\bbQ{\mathbb Q}

\def\bbZ{\mathbb Z}

\def\bfD{\mathbf D}

\begin{document}
\maketitle 

In his proof of the first Brauer-Thrall conjecture \cite{Ro}, Roiter
used an induction scheme which Gabriel formalized in his report on
abelian length categories \cite{G}. The first Brauer-Thrall conjecture
asserts that every finite dimensional algebra of bounded
representation type is of finite representation type. Ringel
noticed\footnote{Cf.\ the footnote on p.\ 91 of \cite{G}.} that the
formalism of Gabriel and Roiter is also useful for studying the
representations of algebras having unbounded representation type. 

In these notes we present a purely combinatorial definition of the
Gabriel-Roiter measure and combine this with an axiomatic
characterization; see also \cite{K}.  Given a finite dimensional
algebra $\La$, the Gabriel-Roiter measure is characterized as a
universal morphism $\ind\La\to P$ of partially ordered sets. The map
is defined on the isomorphism classes of finite dimensional
indecomposable $\La$-modules and is a suitable refinement of the
length function $\ind\La\to\bbN$ which sends a module to its
composition length. The axiomatic treatment is complemented by a
recursive definition of the Gabriel-Roiter measure.

The second part of these notes discusses the Gabriel-Roiter measure
for a fixed abelian length category. This is the original setting for
Gabriel's work. In particular, Gabriel's main property of the measure
is proved. This is used to extend the Gabriel-Roiter measure from
indecomposable to arbitrary objects. Our main example is the category
of finite dimensional $\La$-modules over some finite dimensional
algebra $\La$. We report on Ringel's work \cite{R1,R2}, presenting for
instance his refinement of the first Brauer-Thrall conjecture.

These are the notes for a series of four lectures at the ``Advanced
School and Conference on Representation Theory and Related Topics'' in
Trieste (ICTP, January 2006). I am grateful to the organizers of this
school for exposing me to this beautiful subject. In addition, I wish
to express my thanks to the participants for their enthusiasm and to
Philipp Fahr for numerous helpful comments.

\section{Chains and length functions}

\subsection{The Gabriel-Roiter measure}\label{se:defgr} 
There are a number of possible approaches to define the Gabriel-Roiter
measure. Fix a partially ordered set $(S,\leq)$ which is equipped with
a length function $\la\colon S\to\bbN$. We start off by defining the
Gabriel-Roiter measure for $S$ as a morphism $\m\colon S\to P$ of
partially ordered sets which refines the length function $\la$. Let us
stress right away that the values $\m(x)$ for $x\in S$ are not
relevant. All we need to know is whether for a pair $x,y$ of elements
in $S$, the relation $\m(x)\leq\m(y)$ holds or not.  This is the
essence of a measure and we make this precise in the following
definition.

\begin{defn}
Let $(S,\leq)$ be a partially ordered set. A \emph{measure} $\m$ for
$S$ is a relation on $S$, written $\m(x)\leq \m(y)$, for a pair $x,y$
of elements in $S$, such that for all $x,y,z$ in $S$ the following
holds:
\begin{enumerate}
\item[(M1)] $\m(x)\leq \m(y)$ and $\m(y)\leq \m(z)$ imply $\m(x)\leq \m(z)$.
\item[(M2)] $\m(x)\leq \m(y)$ or $\m(y)\leq \m(x)$.
\item[(M3)] $x\leq y$ implies $\m(x)\leq \m(y)$.
\end{enumerate}
We write $\m(x)=\m(y)$ if both $\m(x)\leq \m(y)$ and $\m(y)\leq \m(x)$
hold.
\end{defn}
A measure $\m$ for $S$ gives rise to an equivalence relation on $S$ as
follows: Call two elements $x$ and $y$ \emph{equivalent} if
$\m(x)=\m(y)$.  The set $S/\m$ of equivalence classes is totally
ordered via $\m$ and the canonical map $S\to S/\m$ is a morphism of
partially ordered sets. Conversely, any morphism $\p\colon S\to P$ to
a totally ordered set $P$ gives rise to a measure $\m$ for $S$
provided one defines $\m(x)\leq \m(y)$ if $\p(x)\leq \p(y)$ holds.

In this section we present three different approaches defining the
Gabriel-Roiter measure for a partially ordered set $S$ and a length
function $\la\colon S\to\bbN$. To be more precise, we define the
Gabriel-Roiter measure as a morphism $S\to\Ch(\bbN)$ of partially
ordered sets, where $\Ch(\bbN)$ denotes the lexicographically ordered
set of finite sets of natural numbers. We complement this by a
recursive and an axiomatic definition.  Note that all three concepts
are equivalent in the sense that they yield the same measure for $S$.

\subsection{The lexicographic order on finite chains}\label{se:lex}
Let $(S,\leq)$ be a partially ordered set.  A subset $X\subseteq S$ is a {\em chain}
if $x_1\leq x_2$ or $x_2\leq x_1$ for each pair $x_1,x_2\in X$.  For a
finite chain $X$, we denote by $\min X$ its minimal and by $\max X$
its maximal element, using the convention
$$\max \emptyset < x< \min \emptyset\quad\text{for all}\quad x\in S.$$ We
write $\Ch(S)$ for the set of all finite chains in $S$  and let
$$\Ch(S,x):=\{X\in\Ch(S)\mid\max X=x\}\quad\text{for}\quad x\in S.$$
On $\Ch(S)$ we consider the {\em lexicographic order} which is defined
by
$$X\leq Y \quad :\Longleftrightarrow\quad \min(Y\setminus X)\leq
\min(X\setminus Y)\quad\text{for}\quad  X,Y\in\Ch(S).$$ 

\begin{rem}
(1) $X\subseteq Y$ implies $X\leq Y$ for $X,Y\in\Ch(S)$.

(2) Suppose that $S$ is totally ordered. Then $\Ch(S)$ is totally
ordered.  We may think of $X\in\Ch(S)\subseteq\{0,1\}^S$ as a string
of $0$s and $1$s which is indexed by the elements in $S$. The usual
lexicographic order on such strings coincides with the lexicographic
order on $\Ch(S)$.
\end{rem}

\begin{exm}
Let $\bbN=\{1,2,3,\cdots\}$ and $\bbQ$ be the set of rational numbers
together with the natural ordering. Then the map
$$\Ch(\bbN)\lto \bbQ ,\quad X\mapsto \sum_{x\in X}2^{-x}$$ is
injective and order preserving, taking values in the interval
$[0,1]$. For instance, the subsets of $\{1,2,3\}$ are ordered as follows:
$$\{\}<\{3\}<\{2\}<\{2,3\}<\{1\}<\{1,3\}<\{1,2\}<\{1,2,3\}.$$
\end{exm}

We need the following properties of the lexicographic order.

\begin{lem}
Let $X,Y\in\Ch(S)$ and $X^*:=X\setminus\{\max X\}$.
\begin{enumerate}
\item $X^*=\max\{X'\in\Ch(S)\mid X'<X\text{ and }\max X'<\max X\}$.
\item If $X^*< Y$ and $\max X\geq\max Y$, then $X\leq Y$.
\end{enumerate}
\end{lem}
\begin{proof}
(1) Let $X'<X$ and $\max X'<\max X$. We show that $X'\leq X^*$. This
    is clear if $X'\subseteq X^*$. Otherwise, we have
$$\min (X^*\setminus X')=\min (X\setminus X')<\min (X'\setminus X)=\min
(X'\setminus X^*),$$ and therefore $X'\leq X^*$.

(2) The assumption $X^*< Y$ implies by definition
$$\min (Y\setminus X^*)<\min (X^*\setminus Y).$$ We consider two
cases. Suppose first that $X^*\subseteq Y$. If $X\subseteq Y$, then
$X\leq Y$. Otherwise, $$\min (Y\setminus X)<\max X=\min (X\setminus Y)$$
and therefore $X<Y$. Now suppose that $X^*\not\subseteq Y$. We use
again that $\max X\geq\max Y$, exclude the case $Y\subseteq X$, and
obtain
$$\min (Y\setminus X)=\min (Y\setminus X^*)<\min (X^*\setminus Y)=\min
(X\setminus Y).$$ Thus $X\leq Y$ and the proof is complete.
\end{proof}

\subsection{Length functions}
Let $(S,\leq)$ be a partially ordered set. A {\em length function} on
$S$ is by definition a map $$\la\colon S\lto\bbN=\{1,2,3,\ldots\}$$
such that $x< y$ in $S$ implies $\la(x)< \la(y)$.  A length function
$\la\colon S\to\bbN$ induces for each $x\in S$ a map
$$\Ch(S,x)\lto\Ch(\bbN,\la(x)),\quad X\mapsto \la(X),$$ and therefore
the following {\em chain length function}
$$S\lto \Ch(\bbN),\quad x\mapsto \la^*(x):=\max\{\la(X)\mid
X\in\Ch(S,x) \}.$$ This chain length function is by definition the
\emph{Gabriel-Roiter measure} for $S$ with respect to $\la$.

We continue with a list of basic properties (C0) -- (C5) of $\la^*$.

\subsection{A recursive definition}
The following property (C0) of the chain length function $\la^*\colon
S\to\Ch(\bbN)$ can be used to define $\la^*$ by induction on the
length of the elements in $S$.  We take this as our second definition
of the Gabriel-Roiter measure for $S$ with respect to $\la$.
Note that $\la^*(x)=\{\la(x)\}$ if $x$ is a minimal element of $S$.
\begin{prop}
Let $x\in S$. 
\begin{enumerate}
\item[(C0)] $\la^*(x)=
\max_{x'<x}\la^*(x')\cup\{\la(x)\}$.
\end{enumerate}
\end{prop}
\begin{proof}
Let $X=\la^*(x)$ and note that $\max X=\la(x)$. The assertion follows from
Lemma~\ref{se:lex} because we have
\begin{equation*}
X\setminus\{\max X\}=\max\{X'\in\Ch(\bbN)\mid X'<X\text{ and }\max
X'<\max X\}.\qedhere
\end{equation*}
\end{proof}

\subsection{Basic properties}\label{se:def}
Let $\la\colon S\to\bbN$ be a length function and $\la^*\colon S\to
\Ch(\bbN)$ the induced chain length function.  The following basic
properties suggest to think of $\la^*$ as a refinement of $\la$.

\begin{prop}
Let $x,y\in S$. 
\begin{enumerate}
\item[(C1)] $x\leq y$ implies $\la^*(x)\leq\la^*(y)$.
\item[(C2)] $\la^*(x)=\la^*(y)$ implies $\la(x)=\la(y)$.
\item[(C3)] $\la^*(x')<\la^*(y)$ for all $x'<x$ and $\la(x)\geq\la(y)$
imply $\la^*(x)\leq\la^*(y)$.
\end{enumerate}
\end{prop}
\begin{proof}
  Suppose $x\leq y$ and let $X\in\Ch(S,x)$. Then
  $Y=X\cup\{y\}\in\Ch(S,y)$ and we have $\la(X)\leq\la(Y)$ since
  $\la(X)\subseteq\la(Y)$. Thus $\la^*(x)\leq \la^*(y)$.  If
  $\la^*(x)= \la^*(y)$, then
    $$\la(x)=\max\la^*(x)=\max\la^*(y)=\la(y).$$ To prove (C3), we use
    (C0) and apply Lemma~\ref{se:lex} with $X=\la^*(x)$ and
    $Y=\la^*(y)$. In fact, $\la^*(x')<\la^*(y)$ for all $x'<x$ implies
    $X^*<Y$, and $\la(x)\geq\la(y)$ implies $\max X\geq \max Y$. Thus
    $X\leq Y$.
\end{proof}

We state some further elementary properties of the map $\la^*$.
\begin{prop}
Let $x,y\in S$. 
\begin{enumerate}
\item[(C4)] $\la^*(x)\leq \la^*(y)$ or $\la^*(y)\leq\la^*(x)$.
\item[(C5)] $\{\la^*(x)\mid x\in S\text{ and }\la(x)\leq n\}$ is
  finite for all $n\in\bbN$.
\end{enumerate}
\end{prop}
\begin{proof}
  (C4) is clear since $\Ch(\bbN)$ is totally ordered. (C5) follows
  from the fact that $\{X\in\Ch(\bbN)\mid \max X\leq n\}$ is finite
  for all $n\in\bbN$.
\end{proof}

The map $\la^*$ induces a measure $\m$ for $S$ in the sense of
Definition~\ref{se:defgr}.

\begin{cor}
The chain length function $\la^*$ induces  via
$$\m(x)\leq \m(y)\quad :\Longleftrightarrow\quad
\la^*(x)\leq\la^*(y)\quad\text{for}\quad x,y\in S$$ a measure for $S$.
Moreover, we have for all $x,y$ in $S$
$$\m(x)= \m(y)\quad \Longleftrightarrow\quad
\max_{x'<x}\m(x')=\max_{y'<y}\m(y')\text{ and }\la(x)=\la(y).$$
\end{cor}
\begin{proof} (C1) and (C4) imply that the map $\la^*$
induces a measure $\m$ for $S$. The characterization for $\m(x)=
\m(y)$ follows from (C0).
\end{proof}

\subsection{An axiomatic definition}\label{se:axiom}
Let $\la\colon S\to\bbN$ be a length function.  We present an
axiomatic characterization of the induced chain length function
$\la^*$.  Thus we can replace the original definition in terms of
chains by three simple conditions which express the fact that $\la^*$
refines $\la$. We take this as our third definition of the
Gabriel-Roiter measure for $S$ with respect to $\la$.

\begin{thm}
Let $\la\colon S\to\bbN$ be a length function. Then there exists a map
$\m\colon S\to P$ into a partially ordered set $P$ satisfying for all
$x,y\in S$ the following:
\begin{enumerate}
\item[(P1)] $x\leq y$ implies $\m(x)\leq\m(y)$.
\item[(P2)] $\m(x)=\m(y)$ implies $\la(x)=\la(y)$.
\item[(P3)] $\m(x')<\m(y)$ for all $x'<x$ and $\la(x)\geq\la(y)$ imply
$\m(x)\leq\m(y)$.
\end{enumerate}
Moreover, for any map $\m'\colon S\to P'$ into a partially ordered set
$P'$ satisfying the above conditions, we have for all $x,y$ in $S$ 
$$\m'(x)\leq\m'(y)\quad\Longleftrightarrow\quad
\m(x)\leq\m(y)\quad\Longleftrightarrow\quad
\la^*(x)\leq\la^*(y).$$ 
\end{thm}
\begin{proof}
We have seen in (\ref{se:def}) that $\la^*$ satisfies (P1) -- (P3).
So it remains to show that for any map $\m\colon S\to P$ into a
partially ordered set $P$, the conditions (P1) -- (P3) uniquely
determine the relation $\m(x)\leq\m(y)$ for any pair $x,y\in S$.  In
fact, we claim that (P1) -- (P3) imply $\m(x)\leq\m(y)$ or
$\m(y)\leq\m(x)$. We proceed by induction on the length of the
elements in $S$. For elements of length $n=1$, the assertion is
clear. In fact, $\la(x)=1=\la(y)$ implies $\m(x)=\m(y)$ by (P3). Now
let $n>1$ and assume the assertion is true for all elements $x\in S$
of length $\la(x)<n$. We choose for each $x\in S$ of length
$\la(x)\leq n$ a {\em Gabriel-Roiter filtration}, that is, a sequence
$$x_1<x_2<\ldots <x_{\g(x)-1}<x_{\g(x)}=x$$ in $S$ such that $x_1$ is
minimal and $\max_{x'<x_{i}}\m(x')=\m(x_{i-1})$ for all $1< i\leq
\g(x)$. Such a filtration exists because the elements $\m(x')$ with
$x'<x$ are totally ordered. Now fix $x,y\in S$ of length at most $n$
and let $I=\{i\geq 1\mid \m(x_i)=\m(y_i)\}$. We consider $r=\max I$
and put $r=0$ if $I=\emptyset$. There are two possible cases.  Suppose
first that $r=\g(x)$ or $r=\g(y)$. If $r=\g(x)$, then
$\m(x)=\m(x_r)=\m(y_r)\leq\m(y)$ by (P1).  Now suppose $\g(x)\neq
r\neq\g(y)$. Then we have $\la(x_{r+1})\neq\la(y_{r+1})$ by (P2) and
(P3). If $\la(x_{r+1})>\la(y_{r+1})$, then we obtain
$\m(x_{r+1})<\m(y_{r+1})$, again using (P2) and (P3). Iterating this
argument, we get $\m(x)=\m(x_{\g(x)})<\m(y_{r+1})$. From (P1) we get
$\m(x)<\m(y_{r+1})\leq\m(y)$. Thus $\m(x)\leq\m(y)$ or
$\m(y)\leq\m(x)$ and the proof is complete.
\end{proof}

\section{Abelian length categories}

\subsection{Additive categories}
A category $\A$ is {\em additive} if every finite family
$X_1,X_2,\ldots,X_n$ of objects has a coproduct
$$X_1\oplus X_2\oplus\ldots \oplus X_n,$$ each set $\Hom_\A(A,B)$
is an abelian group, and the composition maps
$$\Hom_\A(B,C)\times\Hom_\A(A,B)\lto\Hom_\A(A,C)$$ are bilinear. 

\subsection{Abelian categories}
An additive category $\A$ is {\em abelian}, if every map $\p\colon
A\to B$ has a kernel and a cokernel, and if the canonical factorization
$$\xymatrix@=1.5em{\Ker\p\ar[r]^-{\p'}&A\ar[r]^-\p\ar[d]&
B\ar[r]^-{\p''}&\Coker\p\\ 
&\Coker\p'\ar[r]^-{\bar\p}&\Ker\p''\ar[u]}$$
of $\p$ induces  an isomorphism $\bar\p$.

\begin{exm} 
The category of modules over any associative ring is an abelian
category.
\end{exm}

\subsection{Subobjects}
Let $\A$ be an abelian category. We say that two monomorphisms
$X_1\to X$ and $X_2\to X$ are {\em equivalent},
if there exists an isomorphism $X_1\to X_2$ 
making the following diagram commutative.
\begin{equation*}
\xymatrix@=1.0em{X_1\ar[rd]\ar[rr]&&X_2\ar[ld]\\&X}
\end{equation*}
An equivalence class of monomorphisms into $X$ is called a {\em
subobject} of $X$. Given subobjects $X_1\to X$ and
$X_2\to X$, we write $X_1\subseteq X_2$ if there is a
morphism $X_1\to X_2$ making the above diagram commutative. An
object $X\neq 0$ is {\em simple} if $X'\subseteq X$ implies $X'=0$ or
$X'=X$.

\subsection{Length categories}
Let $\A$ be an abelian category.  An object $X$ has {\em finite
length} if it has a finite composition series $$0=X_0\subseteq
X_1\subseteq \ldots \subseteq X_{n-1}\subseteq X_n=X,$$ that is, each
$X_i/X_{i-1}$ is simple.  In this case the length of a composition
series is an invariant of $X$ by the Jordan-H\"older Theorem; it is
called the {\em length} of $X$ and is denoted by $\ell(X)$.  For
instance, $X$ is simple if and only if $\ell(X)=1$. Note that $X$ has
finite length if and only if $X$ is both artinian (i.e.\ satisfies the
descending chain condition on subobjects) and noetherian (i.e.\
satisfies the ascending chain condition on subobjects).

An abelian category is called a {\em length category} if all objects
have finite length and the isomorphism classes of objects form a set.

An object $X\neq 0$ is called {\em indecomposable} if $X=X_1\oplus
X_2$ implies $X_1=0$ or $X_2=0$. A finite length object admits
a finite direct sum decomposition into indecomposable objects having local
endomorphism rings. Moreover, such a decomposition is unique up to an
isomorphism by the Krull-Remak-Schmidt Theorem.

We denote by $\ind\A$ the set of isomorphism classes of indecomposable
objects of $\A$.

\begin{exm}
(1) Let $\La$ be a artinian ring. Then the category of finitely
generated $\La$-modules form a length category which we denote by
$\mod\La$.

(2) Let $k$ be a field and $Q$ be any quiver. Then the finite
    dimensional $k$-linear representations of $Q$ form a length
    category.
\end{exm}

\section{The Gabriel-Roiter measure}

Let $\A$ be an abelian length category. We give the definition of the
Gabriel-Roiter measure for $\A$ which is due to Gabriel \cite{G} and
was inspired by the work of Roiter \cite{Ro}. Then we discuss some specific
properties, including Ringel's results about Gabriel-Roiter inclusions
\cite{R1}.

\subsection{The definition}
Let $\A$ be an abelian length category.  The isomorphism classes of
objects of $\A$ are partially ordered via the subobject relation
$$X\subseteq Y\quad :\Longleftrightarrow\quad \text{there exists a
monomorphism } X\to Y.$$ We consider the length function
$\ell\colon\ind\A\to\bbN$ which takes an object $X$ to its composition
length $\ell(X)$. Then the induced chain length function
$\ell^*\colon\ind\A\to\Ch(\bbN)$ is by definition the {\em
Gabriel-Roiter measure} for $\A$. We will only work with this
definition when making explicit computations. Otherwise, we take the
induced measure in the sense of Definition~\ref{se:defgr} which is
characterized as follows.

\begin{thm}
Let $\A$ be an abelian length category. 
The Gabriel-Roiter measure induces via 
$$\m(X)\leq\m(Y)\quad:\Longleftrightarrow\quad\ell^*(X)\leq\ell^*(Y)\quad\text{for}\quad
X,Y\in\ind\A$$ a relation on $\ind\A$.  This is the unique transitive
relation on $\ind\A$ satisfying for all objects $X,Y$ the following:
\begin{enumerate}
\item[(GR1)] $X\subseteq Y$ implies $\m(X)\leq\m(Y)$.
\item[(GR2)] $\m(X)=\m(Y)$ implies $\ell(X)=\ell(Y)$.
\item[(GR3)] $\m(X')<\m(Y)$ for all $X'\subset X$ and $\ell(X)\geq\ell(Y)$ imply
$\m(X)\leq\m(Y)$.
\end{enumerate}
\end{thm}
Here we use the following convention: We write $\m(X)=\m(Y)$ if
$\m(X)\leq\m(Y)$ and $\m(Y)\leq\m(X)$ hold. Morever, we write
$\m(X)<\m(Y)$ if $\m(X)\leq\m(Y)$ and $\m(X)\neq\m(Y)$ hold.

\begin{proof}
The relation $\m(X)=\m(Y)$ defines an equivalence relation on $\ind\A$
and we denote by $\ind\A/\m$ the set of equivalence classes. This set
is partially ordered via $\m$. The canonical map $\ind\A\to\ind\A/\m$
is a morphism of partially ordered sets satisfying the conditions (P1)
-- (P3) from Theorem~\ref{se:axiom}.  Suppose we have another
transitive relation, written $\m'(X)\leq\m'(Y)$ for $X,Y$ in $\ind\A$,
and satisfying (GR1) -- (GR3). We obtain a second morphism
$\ind\A\to\ind\A/\m'$ of partially ordered sets satisfying the
conditions (P1) -- (P3), and we deduce from Theorem~\ref{se:axiom}
that for all $X,Y$
\begin{equation*}
\m'(X)\leq\m'(Y)\quad\Longleftrightarrow\quad\m(X)\leq\m(Y).\qedhere
\end{equation*}
\end{proof}

\begin{exm}
(1) Let $X\in\A$ be {\em uniserial}, that is, $X$ has a unique composition
series. Then $\ell^*(X)=\{1,2,\ldots,\ell(X)\}$.

(2)  Let $X\in\A$ be an indecomposable object of length at most three. Then
$$\ell^*(X)=\begin{cases} \{1\}&\text{if }\ell(X)=1,\\
\{1,2\}&\text{if }\ell(X)=2,\\
\{1,2,3\}&\text{if }\ell(X)=3\text{ and }\ell(\soc X)=1,\\
\{1,3\}&\text{if }\ell(X)=3\text{ and }\ell(\soc X)\neq 1.\\
\end{cases}$$
Here, $\soc X$ denotes the {\em socle} of $X$, that is, the sum of all simple subobjects.

(3) Let $k$ be a field and consider the category $\A$ of $k$-linear
    representations of the following quiver.
$$1 \longleftarrow 2\longrightarrow 3$$ An indecomposable
representation $V_1 \leftarrow V_2\rightarrow V_3$ is determined by
its dimension vector $(d_1d_2d_3)$, where $d_i=\dim_kV_i$. The
following Hasse diagram displays the partial order on $\ind\A$, where
the layer indicates the length of each object.
$$
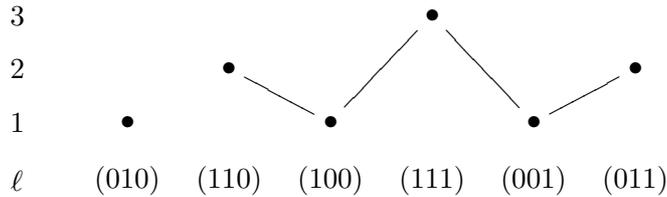
\xymatrix@=0.6em{ 3&&&&&\bullet\ar@{-}[ldd]\ar@{-}[rdd]\\
2&&&\bullet\ar@{-}[rd]&&&&\bullet\ar@{-}[ld]\\
1&&\bullet&&\bullet&&\bullet\\
\ell&&(010)&(110)&(100)&(111)&(001)&(011)\\ }$$ From this diagram one
computes the Gabriel-Roiter measure $\ell^*(X)$ of each indecomposable
object $X$ and obtains the following ordering:
\begin{equation*}
\ell^*(010)=\ell^*(100)=\ell^*(001)=\{1\}<\ell^*(111)=\{1,3\}<\ell^*(110)=\ell^*(011)=\{1,2\}
\end{equation*}
\end{exm}

\subsection{Basic properties}
Recall from (\ref{se:def}) that we have established the
following property of the Gabriel-Roiter measure.
\begin{enumerate}
\item[(GR4)] $\m(X)\leq \m(Y)$ or $\m(Y)\leq\m(X)$ for $X,Y$ in $\ind\A$.
\item[(GR5)] $\{\m(X)\mid X\in \ind\A\text{ and }\ell(X)\leq n\}$ is
  finite for all $n\in\bbN$.
\end{enumerate}

Next we discuss further properties of the Gabriel-Roiter measure which
depend on the fact that $\A$ is a length category.

\subsection{Gabriel-Roiter filtrations}\label{se:filt}
Let $X,Y\in\ind\A$. We say that $X$ is a {\em Gabriel-Roiter
predecessor} of $Y$ if $X\subset Y$ and $\m(X)=\max_{Y'\subset
Y}\m(Y')$. Note that each object $Y\in\ind\A$ which is not simple
admits a Gabriel-Roiter predecessor, by (GR4) and (GR5). A
Gabriel-Roiter predecessor $X$ of $Y$ is usually not unique, but the
value $\m(X)$ is determined by $\m(Y)$.

A sequence $$X_1\subset X_2\subset \ldots\subset X_{n-1} \subset
X_n=X$$ in $\ind\A$ is called a {\em Gabriel-Roiter filtration} of $X$
if $X_1$ is simple and $X_{i-1}$ is a Gabriel-Roiter predecessor of
$X_i$ for all $1< i\leq n$. Clearly, each $X$ admits such a filtration
and the values $\m(X_i)$ are uniquely determined by $X$.

\begin{prop}
Let $X,Y\in\ind\A$.
\begin{enumerate}
\item[(GR6)] $X\in\ind\A$ is simple if and only if $\m(X)\leq\m(Y)$ for
all $Y\in\ind\A$.
\item[(GR7)] Suppose that $\m(X)<\m(Y)$. Then there are $Y'\subset  Y''\subseteq Y$
in $\ind\A$ such that
$Y'$ is a Gabriel-Roiter predecessor of $Y''$ with $\m(Y')\leq\m(X)<\m(Y'')$ and
$\ell(Y')\leq\ell(X)$.
\end{enumerate}
\end{prop}
\begin{proof}
For (GR6), one uses that each indecomposable object has a simple
subobject.  To prove (GR7), fix a Gabriel-Roiter filtration
$Y_1\subset Y_2\subset \ldots \subset Y_n=Y$ of $Y$. We have
$\m(Y_1)\leq\m(X)$ because $Y_1$ is simple. Using (GR4), there exists
some $i$ such that $\m(Y_i)\leq\m(X)<\m(Y_{i+1})$. Now put $Y'=Y_{i}$
and $Y''=Y_{i+1}$. Comparing the filtration of $Y$ with a
Gabriel-Roiter filtration of $X$ (as in the proof of
Theorem~\ref{se:axiom}), we find that $\ell(Y')\leq\ell(X)$.
\end{proof}

\begin{exm}
Let $X\in\A$ be uniserial. Then the composition series is a
Gabriel-Roiter filtration of $X$.
\end{exm}

\subsection{The main property}\label{se:main}
The following main property of the Gabriel-Roiter measure is crucial
for the whole theory.

\begin{prop}[Gabriel]
Let $X,Y_1,\ldots,Y_r\in\ind\A$.
\begin{enumerate}
\item[(GR8)] Suppose that $X\subseteq Y=\oplus_{i=1}^rY_i$. Then
$\m(X)\leq\max\m(Y_i)$ and $X$ is a direct summand of $Y$ if
$\m(X)=\max\m(Y_i)$.
\end{enumerate}
\end{prop}
\begin{proof}
The proof only uses the properties (GR1) -- (GR3) of $\m$.  Fix a
monomorphism $\p\colon X\to Y$. We proceed by induction on
$n=\ell(X)+\ell(Y)$.  If $n= 2$, then $\p$ is an isomorphism and the
assertion is clear. Now suppose $n>2$.  We can assume that
for each $i$ the $i$th component $\p_i\colon X\to Y_i$ of $\p$ is an
epimorphism. Otherwise choose for each $i$ a decomposition
$Y_i'=\oplus_j Y_{ij}$ of the image of $\p_i$ into indecomposables.
Then we use (GR1) and have $\m(X)\leq\max\m(Y_{ij})\leq\max\m(Y_i)$
because $\ell(X)+\ell(Y')<n$ and $Y_{ij}\subseteq Y_i$ for all $j$. Now
suppose that each $\p_i$ is an epimorphism. Thus
$\ell(X)\geq\ell(Y_i)$ for all $i$.  Let $X'\subset X$ be a proper
indecomposable subobject.  Then $\m(X')\leq\max\m(Y_i)$ because
$\ell(X')+\ell(Y)<n$, and $X'$ is a direct summand if
$\m(X')=\max\m(Y_i)$. We can exclude the case that
$\m(X')=\max\m(Y_i)$ because then $X'$ is a proper direct summand of
$X$, which is impossible. Now we apply (GR3) and obtain
$\m(X)\leq\max\m(Y_i)$.  Finally, suppose that
$\m(X)=\max\m(Y_i)=\m(Y_k)$ for some $k$. We claim that we can choose
$k$ such that $\p_k$ is an epimorphism.  Otherwise, replace all $Y_i$
with $\m(X)=\m(Y_i)$ by the image $Y_i'=\oplus_j Y_{ij}$ of $\p_i$ as
before.  We obtain $\m(X)\leq\max\m(Y_{ij})<\m(Y_k)$ since
$Y_{kj}\subset Y_k$ for all $j$, using (GR1) and (GR2).  This is a
contradiction.  Thus $\p_k$ is an epimorphism and in fact an
isomorphism because $\ell(X)=\ell(Y_k)$ by (GR2). In particular, $X$ is
a direct summand of $\oplus_iY_i$. This completes the proof.
\end{proof}

\begin{cor}
Let $X,Y\in\ind\A$ and suppose that $X\subset Y$ with
$\m(X)=\max_{Y'\subset Y}\m(Y')$.  If $X\subseteq U\subset Y$ in $\A$, then
$X$ is a direct summand of $U$.
\end{cor}
\begin{proof}
Let $U=\oplus_i U_i$ be a decomposition into indecomposables. Now apply
(GR8).  We obtain $\m(X)\leq\max\m(U_i)< \m(Y)$ and our assumption
on $X\subset Y$ implies that $X$ is a direct summand of $U$.
\end{proof}

\begin{exm}
(1) Let $Y\in\ind\A$ and suppose that $\m(X)\leq\m(Y)$ for all
$X\in\ind\A$. Then $Y$ is an injective object, because every
monomorphism $Y\to Z$ splits by (GR8). 

(2) Suppose that $\A$ has a cogenerator $Q$, that is, each object in
$\A$ admits a monomorphism into a direct sum of copies of $Q$. Let
$Q=\oplus_i Q_i$ be a decomposition into indecomposable objects.  Then
$\m(X)\leq\max\m(Q_i)$ for all $X\in\ind\A$.
\end{exm}

The Gabriel-Roiter measure $\ell^*\colon\ind\A\to\Ch(\bbN)$ for $\A$
can be extended to a measure defined for all objects in $\A$, not only
the indecomposable ones. Let $X=\oplus_i X_i$ be an object written as
a direct sum of indecomposable objects. Then we define
$$\ell^*(X)=\max\ell^*(X_i).$$

\begin{cor}
The relation 
$$\m(X)\leq\m(Y)\quad:\Longleftrightarrow\quad\ell^*(X)\leq\ell^*(Y)\quad\text{for}\quad
X,Y\in\A$$ induces a measure for the set of isomorphism classes of
$\A$.
\end{cor}
\begin{proof}
We need to verify (M1) -- (M3) from Definition~\ref{se:defgr}. The
first two conditions are automatic and the third is an immediate
consequence of (GR8).
\end{proof}

\subsection{Gabriel-Roiter inclusions}

Let $X,Y\in\ind\A$. An inclusion $X\subseteq Y$ is called {\em Gabriel-Roiter
inclusion} if $\m(X)=\max_{Y'\subset Y}\m(Y')$. Thus we have a Gabriel-Roiter inclusion
$X\subseteq Y$ if and only if $X$ is a Gabriel-Roiter predecessor of $Y$.

\begin{prop}[Ringel]
Let $X,Y\in\ind\A$ and suppose that $X\subset Y$ is a Gabriel-Roiter inclusion.
Then $Y/X$ is an indecomposable object.
\end{prop}
\begin{proof}
Let $Z=Y/X$ and assume that $Z=Z'\oplus Z''$ with $Z''\neq 0$.
We obtain the following commutative diagram with exact rows and columns.
$$\xymatrix{ &&0\ar[d]&0\ar[d]\\
0\ar[r]&X\ar[r]\ar@{=}[d]&Y'\ar[r]\ar[d]&Z'\ar[r]\ar[d]^\inc&0\\
0\ar[r]&X\ar[r]^\inc&Y\ar[d]\ar[r]\ar&Z\ar[d]\ar[r]&0\\
&&Z''\ar[d]\ar@{=}[r]&Z''\ar[d]\\ &&0&0 }$$ We have $X\subseteq
Y'\subset Y$ and therefore the monomorphism $X\to Y'$ splits by
Corollary~\ref{se:main}.  Thus the inclusion $Z'\to Z$
factors through $Y\to Z$ via a split monomorphism $Z'\to Y$. We
conclude that $Z'=0$ since $Y$ is indecomposable.
\end{proof}

\begin{rem}
The argument is borrowed from Auslander and Reiten. They show that the
cokernel of an irreducible monomorphism between indecomposable objects
is indecomposable.
\end{rem}

\begin{cor}
Let $Y$ be an indecomposable object in $\A$ which is not simple.  Then
there exists a short exact sequence $0\to X\to Y\to Z\to 0$ in $\A$ such that
$X$ and $Z$ are indecomposable.
\end{cor}
\begin{proof}
Take $X\subset Y$ with $\m(X)=\max_{Y'\subset Y}\m(Y')$. 
\end{proof}

\section{Finiteness results}

In this section, Ringel's refinement of the first Brauer-Thrall
conjecture is presented \cite{R1}. More precisely, we prove a
structural result about the partial order of the values of the
Gabriel-Roiter measure.

\subsection{Covariant finiteness}\label{se:cov}

A subcategory $\C$ of $\A$ is called {\em covariantly finite} if every
object $X\in\A$ admits a {\em left $\C$-approximation}, that is, a map
$X\to Y$ with $Y\in\C$ such that the induced map
$\Hom_\A(Y,C)\to\Hom_\A(X,C)$ is surjective for all $C\in\C$. We have
also the dual notion: a subcategory $\C$ is {\em contravariantly
finite} if every object in $\A$ admits a {\em right
$\C$-approximation}.

\begin{lem}
Let $\C$ be a subcategory of $\A$ which is closed under taking direct
sums and subobjects. Then $\C$ is a covariantly finite subcategory of
$\A$.
\end{lem}
\begin{proof}
Fix $X\in\A$. Let $X'\subseteq X$ be minimal among the kernels of all
maps $X\to Y$ with $Y\in\C$. Then the canonical map $X\to X/X'$ is a
left $\C$-approximation.
\end{proof}
\begin{rem}
The proof shows that the inclusion functor $\C\to\A$ admits a left
adjoint $F\colon\A\to\C$ which takes $X\in\A$ to $X/X'$. Note that the
adjunction map $X\to FX$ is a left $\C$-approximation.
\end{rem}

Let $M$ be any set of values $\m(X)$. Then we define the subcategory
$$\A(M):=\{X\in\A\mid\m(X)\in M\}.$$

\begin{prop}[Ringel]
Let $M$ be a set of values $\m(X)$ which is closed under predecessors,
that is, $\m(X_1)\leq\m(X_2)$ and $\m(X_2)\in M$ implies $\m(X_1)\in
M$. Then $\A(M)$ is a covariantly finite subcategory of $\A$.
\end{prop}
\begin{proof}
The subcategory $\A(M)$ is closed under taking subobjects by (GR8).
\end{proof}

\subsection{Almost split morphisms}\label{se:ar}

A map $\p\colon X\to Y$ in $\A$ is called {\em left almost split} if
$\p$ is not a split monomorphism and every map $X\to Y'$ in $\A$ which
is not a split monomorphism factors through $\p$. Dually, a map
$\psi\colon Y\to Z$ is called {\em right almost split} if $\psi$ is
not a split epimorphism and every map $Y'\to Z$ which is not a split
epimorphism factors through $\psi$. For example, if $\A=\mod\La$ for
some artin algebra $\La$, then every indecomposable object $X\in\A$
admits a left almost split map starting at $X$ and a right almost
split map ending at $X$; see \cite[Cor.~V.1.17]{ARS}.

\subsection{Immediate successors}\label{se:suc}
Let $X\in\ind\A$. An {\em immediate successor} of $\m(X)$ is by
definition a minimal element in $$\{\m(Y)\mid Y\in\ind\A\text{ and
}\m(X)< \m(Y)\}.$$

\begin{lem}
Let $X,Y\in\ind\A$ and suppose that $X$ is a Gabriel-Roiter predecessor of $Y$.
If $X\to \bar X$ is a left almost split map in $\A$, then $Y$ is a
factor object of $\bar X$.
\end{lem}
\begin{proof}
The monomorphism $X\to Y$ factors through $X\to \bar X$ via a map
$\p\colon \bar X\to Y$.  Let $U$ be the image of $\p$. Applying
Corollary~\ref{se:main}, we find that $U=Y$.
\end{proof}

\begin{prop}
Let $X\in\ind\A$ and suppose there exists $n_X\in\bbN$ such that each
$V\in\ind\A$ with $\m(V)\leq\m(X)$ and $\ell(V)\leq\ell(X)$ admits a
left almost split map $V\to\bar V$ with $\ell(\bar V)\leq n_X$.  Then
there exists an immediate successor of $\m(X)$ provided that $\m(X)$
is not maximal.
\end{prop}
\begin{proof}
Let $\m(X)<\m(Y)$. We apply (GR7) and find $Y'\subset  Y''\subseteq Y$ in $\ind\A$
such that $Y'$ is a Gabriel-Roiter predecessor of $Y''$ with
$\m(Y')\leq\m(X)<\m(Y'')\leq \m(Y)$ and $\ell(Y')\leq\ell(X)$. The
preceding lemma implies $\ell(Y'')\leq n_X$, and (GR5) implies that the
number of values $\m(Y'')$ is finite. Thus there exists a minimal
element among those $\m(Y'')$.
\end{proof}

\begin{cor}[Ringel]
Let $\La$ be an artin algebra and $X\in\ind\La$. Then there exists an
immediate successor of $\m(X)$ provided that $\m(X)$ is not maximal.
\end{cor}
\begin{proof}
Use that there exists $n_\La\in\bbN$ having the following property:
for each indecomposable $V\in\mod\La$, there exists a left almost
split map $V\to\bar V$ satisfying $\ell(\bar V)\leq n_\La\ell(V)$.  In
fact, one takes $n_\La=pq$, where $p$ denotes the maximal length of an
indecomposable projective $\La$-module and $q$ denotes the maximal
length of an indecomposable injective $\La$-module; see
\cite[Prop.~V.6.6]{ARS}.
\end{proof}

\subsection{A finiteness criterion}\label{se:fin}

We present a criterion for a subcategory $\C$ of $\A$ such that the
number of indecomposable objects in $\C$ is finite. This is based on
the following classical lemma.

\begin{lem}[Harada-Sai]
Let $n\in\bbN$. A composition $X_1\to X_2\to \ldots\to X_{2^n}$ of
non-invertible maps between indecomposable objects of length at most
$n$ is zero.
\end{lem}
\begin{proof}
See \cite[Cor.~VI.1.3]{ARS}.
\end{proof}

\begin{prop} 
Let $\A$ be a length category with left almost split maps and only
finitely many isomorphism classes of simple objects. Suppose that $\C$
is a subcategory such that
\begin{enumerate}
\item $\C$ is covariantly finite, and
\item there exists $n\in\bbN$ such that $\ell(X)\leq n$ for all
indecomposable $X\in\C$.
\end{enumerate}
Then there are only finitely many isomorphism classes of indecomposable
objects in $\C$.
\end{prop}
\begin{proof}
We claim that we can construct all indecomposable objects $X\in\C$ in
at most $2^n$ steps from the finitely many simple objects in $\A$ as
follows. Choose a non-zero map $S\to X$ from a simple object $S$ and
factor this map through the left $\C$-approximation $S\to S'$. Take an
indecomposable direct summand $X_0$ of $S'$ such that the component
$S\to X_0\to X$ of the composition $S\to S'\to X$ is non-zero. Stop if
$X_0\to X$ is an isomorphism. Otherwise take a left almost split map
$X_0\to Y_0$ and a left $\C$-approximation $Y_0\to Z_0$. The map
$X_0\to X$ factors through the composition $X_0\to Y_0\to Z_0$ and we
choose an indecomposable direct summand $X_1$ of $Z_0$ such that the
component $X_0\to Y_0\to X_1\to X$ is non-zero. Again, we stop if
$X_1\to X$ is an isomorphism. Otherwise, we continue as before and
obtain in step $r$ a sequence of non-invertible maps $$X_0\to X_1\to
X_2\to \ldots\to X_r$$ such that the composition is non-zero. The
Harada-Sai lemma implies that $r<2^n$ because $\ell(X_i)\leq n$ for
all $i$ by our assumption. Thus $X$ is isomorphic to $X_i$ for some
$i<2^n$, and we obtain $X$ in at most $2^n$ steps, having in each step
only finitely many choices by taking an indecomposable direct
summand. We conclude that $\C$ has only a finite number of
indecomposable objects.
\end{proof}

\begin{rem}
This classical argument provides a quick proof of the first
Brauer-Thrall conjecture; it is due to Auslander and Yamagata.
\end{rem}

\subsection{The initial segment}\label{se:init}
\begin{thm}[Ringel]
Let $\A$ be a length category such that $\ind\A$ is infinite. Suppose
also that $\A$ has only finitely many isomorphism classes of simple
objects and that every indecomposable object admits a left almost
split map. Then there exist infinitely many values
$\m(X_1)<\m(X_2)<\m(X_3)<\ldots$ of the Gabriel-Roiter measure for
$\A$ having the following properties.
\begin{enumerate}
\item If $\m(X)\neq\m(X_i)$ for all $i$, then $\m(X_i)<\m(X)$ for all $i$.
\item The set $\{X\in\ind\A\mid \m(X)=\m(X_i)\}$ is finite for all $i$.
\end{enumerate}
\end{thm}
\begin{proof}
We construct the values $\m(X_i)$ by induction as follows.  Take for
$X_1$ any simple object. Observe that $\m(X_1)$ is minimal among all
$\m(X)$ by (GR6) and that only finitely many $X\in\ind\A$ satisfy
$\m(X)=\m(X_1)$ because $\A$ has only finitely many simple objects.
Now suppose that $\m(X_1)<\ldots<\m(X_n)$ have been constructed,
satisfying the conditions (1) and (2) for all $1\leq i\leq n$.  We can
apply Proposition~\ref{se:suc} and find an immediate successor
$\m(X_{n+1})$ of $\m(X_n)$. It remains to show that the set
$\{X\in\ind\A\mid \m(X)=\m(X_{n+1})\}$ is finite.  To this end
consider $M=\{\m(X_1),\ldots,\m(X_{n+1})\}$. We know from
Proposition~\ref{se:cov} that $\A(M)$ is a covariantly finite
subcategory. Clearly, $\ell(X)$ is bounded by
$\max\{\ell(X_i),\ldots,\ell(X_{n+1})\}$ for all indecomposable
$X\in\A(M)$ by (GR2). We conclude from Proposition~\ref{se:fin} that
the number of indecomposables in $\A(M)$ is finite. Thus
$\{X\in\ind\A\mid \m(X)=\m(X_{n+1})\}$ is finite and the proof is
complete.
\end{proof}

\begin{cor}[Brauer-Thrall I]
Let $\A$ be a length category satisfying the above conditions. Then for
every $n\in\bbN$ there exists an indecomposable object $X\in\A$ with
$\ell (X)> n$.
\end{cor}
\begin{proof}
Use that for fixed $n\in\bbN$, there are only finitely many values
$\m(X)$ with $\ell(X)\leq n$, by (GR5).
\end{proof}

\subsection{The terminal segment}\label{se:term}
\begin{thm}[Ringel]
Let $\A$ be a length category such that $\ind\A$ is infinite. Suppose
also that $\A$ has a cogenerator (i.e.\ an object $Q$ such that each
object in $\A$ admits a monomorphism into a direct sum of copies of
$Q$) and that every indecomposable object admits a right almost split
map. Then there exist infinitely many values
$\m(X^1)>\m(X^2)>\m(X^3)>\ldots$ of the Gabriel-Roiter measure for $\A$
having the following properties.
\begin{enumerate}
\item If $\m(X)\neq\m(X^i)$ for all $i$, then $\m(X^i)>\m(X)$ for all $i$.
\item The set $\{X\in\ind\A\mid \m(X)=\m(X^i)\}$ is finite for all $i$.
\end{enumerate}
\end{thm}
The proof is based on the following lemma.

\begin{lem}[Auslander-Smal{\o}]
Let $\A$ be a length category and let $X\in\A$. Denote by $\A_X$ the
subcategory formed by all objects in $\A$ having no indecomposable
direct summand which is isomorphic to a direct summand of $X$.  If
every indecomposable direct summand of $X$ admits a right almost split
map, then $\A_X$ is contravariantly finite.
\end{lem}
\begin{proof}
Let $X=\oplus^r_{i_0=1}X_{i_0}$ be a decomposition into
indecomposables. It is sufficient to construct a right
$\A_X$-approximation for each indecomposable object $Z\in\A$. We take
the identity map if $Z\in\A_X$. Otherwise, $Z$ is isomorphic to
$X_{i_0}$ for some $i_0$ and we proceed as follows. Let
$\p_{i_0}\colon\bar X_{i_0}\to X_{i_0}$ be a right almost split map
and choose a decomposition $$\bar X_{i_0}=Y_{i_0}\oplus
(\oplus_{i_1}X_{i_0i_1})$$ such that $Y_{i_0}\in\A_X$ and
$i_0i_1\in\{1,\ldots,r\}$ for all $i_1$.  Note that each map $V\to
X_{i_0}$ with $V\in\A_X$ factors through $\p_{i_0}$.  Also, each
component $X_{i_0i_1}\to X_{i_0}$ of $\p_{i_0}$ is non-invertible. Now
compose $\p_{i_0}$ with $\id_{Y_{i_0}}\oplus(\oplus_{i_1}\p_{i_0i_1})$
to obtain a map
$$Y_{i_0}\oplus(\oplus_{i_1}(Y_{i_0i_1}\oplus(\oplus_{i_2}X_{i_0i_1i_2})))\to
Y_{i_0}\oplus(\oplus_{i_1}X_{i_0i_1})\to X_{i_0}.$$ Again, each map
$V\to X_{i_0}$ with $V\in\A_X$ factors through this new map, and each
component $X_{i_0i_1i_2}\to X_{i_0i_1}$ is non-invertible.  We
continue this procedure, compose this map with
$$\id_{Y_{i_0}}\oplus(\oplus_{i_1}(\id_{Y_{i_0i_1}}
\oplus(\oplus_{i_2}\p_{i_0i_1i_2}))),$$ and so on.
Now let $n=2^m$ where $m=\max\{\ell(X_1),\ldots,\ell(X_r)\}$.
Then the Harada-Sai lemma implies that any composition
$$X_{i_0i_1\ldots i_n}\to X_{i_0i_1\ldots i_{n-1}}\to \ldots\to
X_{i_0i_1}\to X_{i_0}$$ is zero. Thus the induced map
$$\oplus_{j=0}^n(\oplus_{i_1,i_2,\ldots,i_j}Y_{i_0i_1\ldots i_j})\lto X_{i_0}$$
is a right $\A_X$-approximation of $X_{i_0}$.
\end{proof}

\begin{proof}[Proof of the theorem]
We construct the values $\m(X^i)$ by induction as follows. Let $n\geq
0$ and suppose that $\m(X^1)>\ldots>\m(X^n)$ have been constructed,
satisfying the conditions (1) and (2) for all $1\leq i\leq n$.  Denote
by $P$ the direct sum of all $X\in\ind\A$ with $\m(X)\geq\m(X^n)$, and
let $P=0$ if $n=0$. Choose a right $\A_P$-approximation $P'\to Q$ and
take for $X^{n+1}$ any indecomposable direct summand $X$ of $P'$ such
that $\m(X)$ is maximal. Observe that every indecomposable object
$X\in\A_P$ is cogenerated by $Q$ and therefore by $P'$. Thus (GR8)
implies that $\m(X)$ is bounded by $\m(X^{n+1})$. Moreover, if
$\m(X)=\m(X^{n+1})$, then $X$ is isomorphic to a direct summand of
$P'$. Thus $\{X\in\ind\A\mid \m(X)=\m(X^{n+1})\}$ is finite and the
proof is complete.
\end{proof}

Let $\La$ be an artin algebra of infinite representation type. Then
$\A=\mod\La$ satisfies the assumptions of Theorems~\ref{se:init} and
\ref{se:term}.  Let us summarize the structure of the partial order on
the values of the Gabriel-Roiter measure as follows. We have
$$\ind\A/\m:=\{\m(X)\mid X\in\ind\A\}=S_\init\sqcup S_\cent\sqcup
S_\term\cong\bbN\sqcup S_\cent \sqcup\bbN^\op,$$ where the notation
$S=S_1\sqcup S_2$ for a poset $S$ means $S=S_1\cup S_2$ and $x_1<x_2$
for all $x_1\in S_1$, $x_2\in S_2$.

\subsection{The Kronecker algebra}
Let $\La=\smatrix{k&k^2\\ 0&k}$ be the Kronecker algebra over an
algebraically closed field $k$.  We consider the abelian length
category which is formed by all finite dimensional $\La$-modules. A
complete list of indecomposable objects is given by the preprojectives
$P_n$, the regulars $R_n(\a,\b)$, and the preinjectives $Q_n$; see
\cite[Thm.~VIII.7.5]{ARS}. More precisely,
$$\ind\La=\{P_n\mid n\in\bbN\}\cup\{R_n(\a,\b)\mid
n\in\bbN,\,(\a,\b)\in\mathbb P_k^1\}\cup\{Q_n\mid n\in\bbN\},$$ and we
obtain the following Hasse diagram.
$$
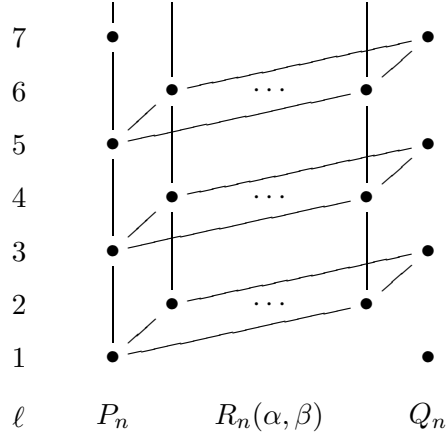
\xymatrix@=0.6em{
&&\ar@{-}[d]&\ar@{-}[dd]&&\ar@{-}[dd]\\
7&&\bullet\ar@{-}[dd]&&&&\bullet\ar@{-}[dl]\ar@{-}[dlll]\\
6&&&\bullet\ar@{-}[dd]\ar@{-}[dl]&\cdots&\bullet\ar@{-}[dd]\ar@{-}[dlll]\\
5&&\bullet\ar@{-}[dd]&&&&\bullet\ar@{-}[dl]\ar@{-}[dlll]\\
4&&&\bullet\ar@{-}[dd]\ar@{-}[dl]&\cdots&\bullet\ar@{-}[dd]\ar@{-}[dlll]\\
3&&\bullet\ar@{-}[dd]&&&&\bullet\ar@{-}[dl]\ar@{-}[dlll]\\
2&&&\bullet\ar@{-}[dl]&\cdots&\bullet\ar@{-}[dlll]\\
1&&\bullet&&&&\bullet\\ \ell&&P_n&&R_n(\a,\b)&&Q_n }$$ The set of
indecomposables is ordered via the Gabriel-Roiter measure as follows:
\begin{multline*}
\m(Q_1)=\m(P_1)<\m(P_2)<\m(P_3)<\ldots \;\;\;<\m(R_1)<\m(R_2)<\m(R_3)<\ldots\\
\ldots<\m(Q_4)<\m(Q_3)<\m(Q_2)
\end{multline*}

\section{The Gabriel-Roiter measure for derived categories}

Let $\A$ be an abelian length category. We propose a definition of the
Gabriel-Roiter measure for the bounded derived category
$\bfD^b(\A)$. The derived Gabriel-Roiter measure extends the
Gabriel-Roiter measure for the underlying abelian category $\A$.

\subsection{The definition}
The bounded derived category $\bfD^b(\A)$ of $\A$ is by definition the
full subcategory of the derived category $\bfD(\A)$ which is formed by
all complexes $X$ such that $H^nX=0$ for almost all $n$. Note that each
object of $\bfD^b(\A)$ admits a finite direct sum decomposition into
indecomposable objects having local endomorphism rings. Moreover, such
a decomposition is unique up to an isomorphism.  We denote by
$\ind\bfD^b(\A)$ the set of isomorphism classes of indecomposable
objects of $\bfD^b(\A)$.

We consider the functor
$$\bfD^b(\A)\lto\A,\quad X\mapsto H^*X=\oplus_{n\in{\bbZ}}H^nX,$$
and the isomorphism classes of objects of $\bfD^b(\A)$ are partially
ordered via
$$X\leq Y\quad :\Longleftrightarrow\quad \begin{cases}\text{there
exists a map } X\to Y \text{ inducing}\\ \text{a monomorphism }
H^*X\to H^*Y.\end{cases}$$ We have the length function
$$\ell_{H^*}\colon\ind\bfD^b(\A)\lto\bbN,\quad X\mapsto\ell(H^*X)$$
and the induced chain length function
$\ell_{H^*}^*\colon\ind\bfD^b(\A)\to\Ch(\bbN)$ is by definition the {\em
Gabriel-Roiter measure} for $\bfD^b(\A)$.

\subsection{Derived versus abelian Gabriel-Roiter measure}
\begin{prop}
The Gabriel-Roiter measure for $\bfD^b(\A)$ extends the Gabriel-Roiter
measure for $\A$. More precisely, the canonical functor
$\A\to\bfD^b(\A)$ sending an object of $\A$ to the corresponding
complex concentrated in degree zero induces an inclusion
$\ind\A\to\ind\bfD^b(\A)$ of partially ordered sets, which makes the
following diagram commutative.
$$\xymatrix{\ind\A\ar[rd]_-{\ell^*}\ar[rr]^-\inc&&\ind\bfD^b(\A)
\ar[ld]^-{\ell^*_{H^*}}\\&\Ch(\bbN) }$$
\end{prop}
\begin{proof}
Use the fact that the diagram
$$\xymatrix{\ind\A\ar[rd]_-{\ell}\ar[rr]^-\inc&&\ind\bfD^b(\A)
\ar[ld]^-{\ell_{H^*}}\\&\bbN }$$
is commutative and that $\ind\A$ is closed under predeccessors in
$\ind\bfD^b(\A)$.
\end{proof}

\subsection{An alternative definition}
For an alternative definition of the Gabriel-Roiter measure for
$\bfD^b(\A)$, consider the lexicographic order on
$$\coprod_\bbZ\bbN_0:=\{(x_n)\in\prod_\bbZ\bbN_0\mid x_n=0\text{ for
almost all }n\}, \text{ with}$$
$$(x_n)\leq (y_n)\quad :\Longleftrightarrow\quad \begin{cases}x_i=y_i\text{
for all }i\in\bbZ, \text{ or}\\
x_i\leq y_i\text{ for }i=\min\{n\in\bbZ\mid x_n\neq y_n\}.\end{cases}$$
Take instead of $\ell_{H^*}$ the length function
$$\la\colon\ind\bfD^b(\A)\lto\coprod_\bbZ\bbN_0,\quad X\mapsto(\ell(H^nX)),$$
and instead of $\ell^*_{H^*}$ the induced chain length function
$$\la^*\colon\ind\bfD^b(\A)\lto\Ch(\coprod_\bbZ\bbN_0).$$ We
illustrate the difference between both definitions by taking a
hereditary length category $\A$. Recall that $\A$ is \emph{hereditary}
if $\Ext^2_\A(-,-)=0$.  Then each indecomposable object of
$\bfD^b(\A)$ is isomorphic to a complex concentrated in a single
degree. Identifying objects having the same Gabriel-Roiter measure, we obtain
$$\ind\bfD^b(\A)/\ell_{H^*}^*=\ind\A/\ell^*,$$ whereas
$$\ind\bfD^b(\A)/\la^*=
\quad\ldots\sqcup\ind\A/\ell^*\sqcup\ind\A/\ell^*
\sqcup\ind\A/\ell^*\sqcup\ldots\,.$$

\end{document}